\documentclass[
final
]{dmtcs-episciences}

\usepackage[figurename=Fig.]{caption}

\usepackage{amsmath,amssymb}
\usepackage{amsthm}

\newtheorem{theorem}{Theorem}
\newtheorem{lemma}[theorem]{Lemma}
\newtheorem{corollary}[theorem]{Corollary}
\newtheorem{proposition}[theorem]{Proposition}

\usepackage[utf8]{inputenc}
\usepackage{subfigure}

\usepackage{graphicx}

\author{Kengo Enami\affiliationmark{1}\thanks{This work was supported by JSPS KAKENHI Grant Number JP19J13359.}}
\title[Embeddings of 3-connected 3-regular planar graphs]{Embeddings of 3-connected 3-regular planar graphs on surfaces of non-negative Euler characteristic}
\affiliation{
Graduate School of Environment and Information Sciences, Yokohama National University, Yokohama , Japan}

\received{2018-7-3}
\revised{2019-5-16, 2019-9-6}
\accepted{2019-9-6}
\begin{document}
\publicationdetails{21}{2019}{4}{11}{4656}
\maketitle
\begin{abstract}
  Whitney's theorem states that every 3-connected planar graph is uniquely embeddable on the sphere.
On the other hand, it has many inequivalent embeddings on another surface.
We shall characterize structures of a $3$-connected $3$-regular planar graph $G$ embedded on the projective-plane, the torus and the Klein bottle,
and give a one-to-one correspondence between inequivalent embeddings of $G$ on each surface and some subgraphs of the dual of $G$ embedded on the sphere.
These results enable us to
give explicit bounds for the number of inequivalent embeddings of $G$ on each surface,
and propose effective algorithms for enumerating and counting these embeddings.
\end{abstract}

%
%

\section{Introduction}
\label{sec:in}

An \emph{embedding} of a graph $G$ on a surface $F^2$,
which is a compact 2-dimensional manifold without boundary,
is a drawing of $G$ on $F^2$ without edge crossing.
Technically, we regard an embedding as an injective continuous map $f:G\to F^2$,
while we often consider that
$G$ is already mapped on a surface and denote its image by $G$ itself.
The \emph{faces} are the connected components of the open set $F^2-f(G)$.
In this paper, we focus on only finite, undirected and simple graphs.
Moreover, we assume that embeddings are \emph{cellular},
that is, each face must be homeomorphic to an open 2-cell,
which contains neither handles nor crosscaps.
For terminologies of topological graph theory, we refer to \cite{tex1,tex2}.

Two embeddings $f_1,f_2:G\to F^2$ are \emph{equivalent}
if there is a homeomorphism $h:F^2\to F^2$ such that $hf_1=f_2$.
We say that a graph $G$ is \emph{uniquely embeddable} on $F^2$ (up to equivalence)
if any two embeddings of $G$ on $F^2$ are equivalent.
The following two questions are important and have attracted many topological graph theorists:
(1) What kind of structures generates inequivalent embeddings of a given graph?
(We often call such a structure the \emph{re-embedding structure}.)
(2) How many inequivalent embeddings on a fixed surface does a graph have?

These problems were first studied by
Whitney \cite{Whitney1,Whitney2}.
He proved that
one of any two embeddings of a 2-connected planar graph on the sphere can be obtained from the other by a sequence of simple local re-embeddings, called Whitney's 2-flipping
(see \cite[Sections~2 and 5]{tex2} for details), and
every 3-connected planar graph is uniquely embeddable on the sphere.

With regard to the uniqueness of a graph embedded on a non-spherical surface,
its ``face-width" plays an effective role.
The \emph{face-width} $\textbf{fw}(G)$ of a graph $G$ embedded on a non-spherical surface $F^2$ is defined by $$\textbf{fw}(G)=\min \{ |G\cap \ell | : \ell \mbox{ is a noncontractible simple closed curve on } F^2 \} .$$
In \cite{Mohar4,rep,Seymour},
it was proved that a $3$-connected graph embedded on a non-spherical surface with sufficiently high face-width is uniquely embeddable on this surface.
See some other studies
\cite{Mohar1,Robertson2}
for relations between
the number of inequivalent embeddings of a graph on a fixed surface and
its face-width.

Not only the face-width
but also the connectivity
has a strong relation to the number of inequivalent embeddings on surfaces with lower genera,
which can be expected from Whitney's theorem.
Negami \cite{Negami2} proved that every $6$-connected toroidal graph except for three graphs
is uniquely embeddable on the torus.
Kitakubo and Negami \cite{Negami1}, and
Suzuki \cite{Suzuki} studied the number of inequivalent embeddings of $5$-connected and $4$-connected non-planar graphs on the projective-plane, respectively.
Recently, Maharry et al. \cite{Robertson1} constructed re-embedding structures of non-planar graphs on the projective-plane completely and pointed out some mistakes in the past studies.
In these papers, they analysed re-embedding structures of ``non-planar" graphs.

On the other hand,
Mohar et al. \cite{Mohar2,Mohar3}
showed that 2-connected ``planar" graphs embedded on non-spherical surfaces have special re-embedding structures,
called ``patch structures",
while they have not given specific structures on each surface except for the projective-plane
and not mentioned the number of inequivalent embeddings on any surface. 

In Section \ref{sec2}, we shall construct the complete list of re-embedding structures of a planar graph $G$ embedded on the projective-plane, the torus or the Klein bottle
when $G$ is $3$-connected and 3-regular.
In this argument, $3$-connectivity is essential
in order for us to ignore Whitney's $2$-flippings, and when a $3$-connected planar graph $G$ is $3$-regular, we can describe re-embedding structures of $G$ on each surface completely.
These re-embedding structures lead to the following results.

We denote the complete graph of order $n$ by $K_n$
and for a positive integer $k\geq 2$,
a complete $k$-partite graph with $k$ partite sets $V_1,V_2,\cdots , V_k$ such that $|V_i|=n_i$ for $1\leq i\leq k$ by $K_{n_1,n_2,\ldots ,n_k}$.
Note that a complete $k$-partite graph $K_{1,1,\ldots ,1}$ is isomorphic to the complete graph $K_k$ of order $k$.

\begin{theorem}\label{th1}
There exists a one-to-one correspondence between inequivalent embeddings of a $3$-connected $3$-regular planar graph on the projective-plane and subgraphs of the dual graph of the graph embedded on the sphere isomorphic to $K_2$ or $K_4$.
\end{theorem}

\begin{theorem}\label{th0}
There exists a one-to-one correspondence between inequivalent embeddings of a $3$-connected $3$-regular planar graph on the torus and subgraphs of the dual graph of the graph embedded on the sphere isomorphic to $K_{2,2,2}$, $K_{2,2m}$ or $K_{1,1,2m-1}$ for some positive integer $m$.
\end{theorem}

\begin{theorem}\label{th2}
There exists a one-to-one correspondence between inequivalent embeddings of a $3$-connected $3$-regular planar graph on the Klein bottle and subgraphs of the dual graph of the graph embedded on the sphere isomorphic to $K_{2,2m-1}$ or $K_{1,1,2m}$ for some positive integer $m$, or one of the six graphs $A_1$ to $A_6$ shown in Fig.~$\ref{HX}$.
\end{theorem}

\begin{figure}[htbp]
\begin{center}
\includegraphics[width=80mm]{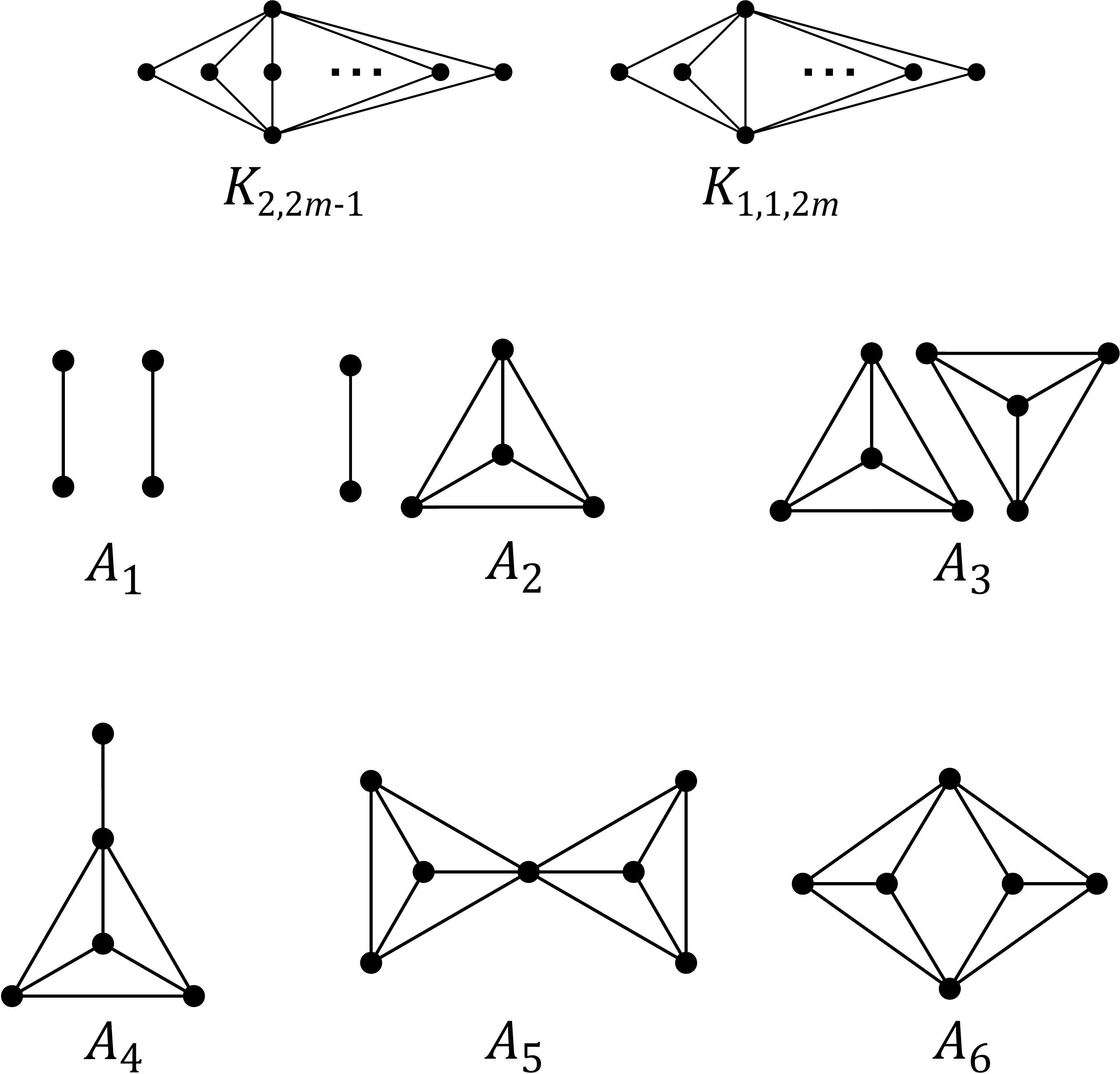}
\end{center}
\caption{The eight graphs}
\label{HX}
\end{figure}

Based on these theorems,
in Section \ref{sec3}, we will give explicit bounds for the number of
inequivalent embeddings of a 3-connected 3-regular planar graph $G$ on each of the projective-plane, the torus and the Klein bottle.
In addition, we will propose effective algorithms for enumerating and counting these embeddings.
In particular, even though $G$ may have exponentially many inequivalent embeddings on the torus and the Klein bottle,
we can calculate the total number of such embeddings in polynomial time.

These results allow the computation and the study of the \emph{genus distributions} of a large family of graphs.
We denote the number of inequivalent embeddings of a graph $G$ on the orientable surface of genus $k$ (resp. the non-orientable surface of genus $h$) by $g_G(k)$ (resp. $\tilde{g}_G(h)$).
The genus distribution (resp. non-orientable genus distribution) of $G$ is defined as the sequence $g_G(0), g_G(1), g_G(2), \cdots$
(resp. $\tilde{g}_G(0), \tilde{g}_G(1), \tilde{g}_G(2), \cdots$).
The topic of genus distributions was introduced by Gross and Furst \cite{Gross1}
and studied in various papers.
Whether the genus distribution of every graph is log-concave is an interesting problem conjectured in \cite{Gross2},
and still remains to be solved.
From the genus distribution's point of view,
we give explicit bounds for $g_G(1),\tilde{g}_G(1)$ and $\tilde{g}_G(2)$ of a $3$-connected $3$-regular planar graph $G$
and algorithms for calculating them.

First of all, we focus on facial cycles in a $3$-connected $3$-regular planar graph embedded on the sphere in Section \ref{sec1}.

\section{Re-embeddings of planar graphs with twisted edges}\label{sec1}

\subsection{Rotation systems and embedding schemes}

First, we introduce two combinatorial ways of describing embeddings of a graph;
``rotation systems" and ``embedding schemes".
A general description on the rotation system and the embedding scheme can be found in \cite{tex2}.

Suppose that a connected graph $G$ is embedded on an orientable surface.
A \emph{rotation} $\rho_v$ around a vertex $v$ of $G$ is a cyclic permutation of edges incident with a vertex $v$
such that $\rho_v(e)$ is the successor of $e$ in the clockwise ordering around $v$.
A \emph{rotation system} for the embedded graph $G$ is the collection of $\rho_v$, denoted by $\rho =\{ \rho_v : v\in V(G)\}$.
It is well-known that every embedding of a connected graph on an orientable surface is uniquely determined up to equivalence by its rotation system.
Moreover, there are no rotation systems representing this embedding other than this rotation system and its inverse.

On the other hand, in order to include embeddings on nonorientable surfaces,
we have to add the following concept.
Let $f(G)$ be another embedding of $G$ on a surface,
which is not necessarily orientable.
There are two possible cyclic ordering of edges incident with each vertex $v$ of $f(G)$.
Choose one of them and denote it by $\rho_v$.
A closed walk $C$ in an embedded graph is
\emph{facial}
if $C$ bounds a face of the graph.
A \emph{signature} of $E(G)$ is a map outputting $1$ or $-1$ from each edge of $G$,
denoted by $\lambda$,
such that for an edge $e=uv$ with its endvertices $u$ and $v$,
$\lambda(e)=1$
if a subwalk induced by the three edges $\rho_u(e)$, $e$ and $\rho_v^{-1}(e)$
is included in a facial walk,
otherwise $\lambda(e)=-1$.
It can be shown that this definition of the signature $\lambda$ is consistent,
that is, $\lambda(uv)=\lambda(vu)$ for every edge $uv$.
The pair $(\rho , \lambda )$,
where $\rho =\{ \rho_v : v\in V(G)\}$ is obtained by the above procedure,
is called an \emph{embedding scheme} for $f(G)$.
An embedding scheme determines exactly one embedding of $G$.
Unfortunately,
an embedding scheme representing a given embedding of $G$ is not uniquely determined, unlike rotation systems for the orientable case.
For an embedded graph $G$
associated with a given embedding scheme $(\rho , \lambda )$,
an edge $e$ with $\lambda(e)=-1$ is called \emph{twisted}.
If there are no twisted edges
then $G$ is embedded on an orientable surface obtained by the rotation system $\rho$.

Hereafter, suppose that $G$ is 3-connected, 3-regular and planar.
In addition, we assume that $G$ is already embedded on the sphere with its rotation system $\rho = \{ \rho_v : v\in V(G)\}$.
(By Whitney's theorem, $G$ is uniquely embeddable on the sphere.)
Since $G$ is 3-regular,
there are only two possible rotations around each vertex of $G$,
and one of them is the inverse of the other.
This implies that for any embedding $f(G)$ of $G$ on any surface,
we can choose $\rho _v$ as the local rotation around each vertex $v$.
Thus, $f(G)$ can be determined by an embedding scheme $(\rho , \lambda)$
with a suitable signature $\lambda$.
We denote the set of twisted edges associated with this embedding scheme by $X$.
In this situation,
we regard this embedding as a re-embedding of $G$
obtained by twisting all edges of $X$
and denote it by $f_X(G)$.
In addition, let $F^2_X$ be the surface
where $f_X(G)$ is embedded.

\subsection{Facial cycles in planar graphs}

Choose two distinct subsets $X_1$ and $X_2$ of $E(G)$.
Then, there is an edge $e$ belonging to only one of either $X_1$ or $X_2$.
We may assume that $e\in X_2$.
It is easy to check that every facial walk of a 3-connected planar graph embedded on the sphere is a cycle.
Thus, there are exactly two facial cycles containing $e$ of $G$,
denoted by $C$ and $C'$.
Let $e_1$ and $e_2$ be the edges of $C'$ adjacent to $e$.
Fig.~\ref{arounde} presents local neighbourhoods around $f_{X_1}(e)$ and $f_{X_2}(e)$.
Note that $f_{X_1}(C)$ and $f_{X_2}(C)$ are drawn by bold lines in Fig.~\ref{arounde}.

\begin{figure}[htbp]
\begin{center}
\includegraphics[width=80mm]{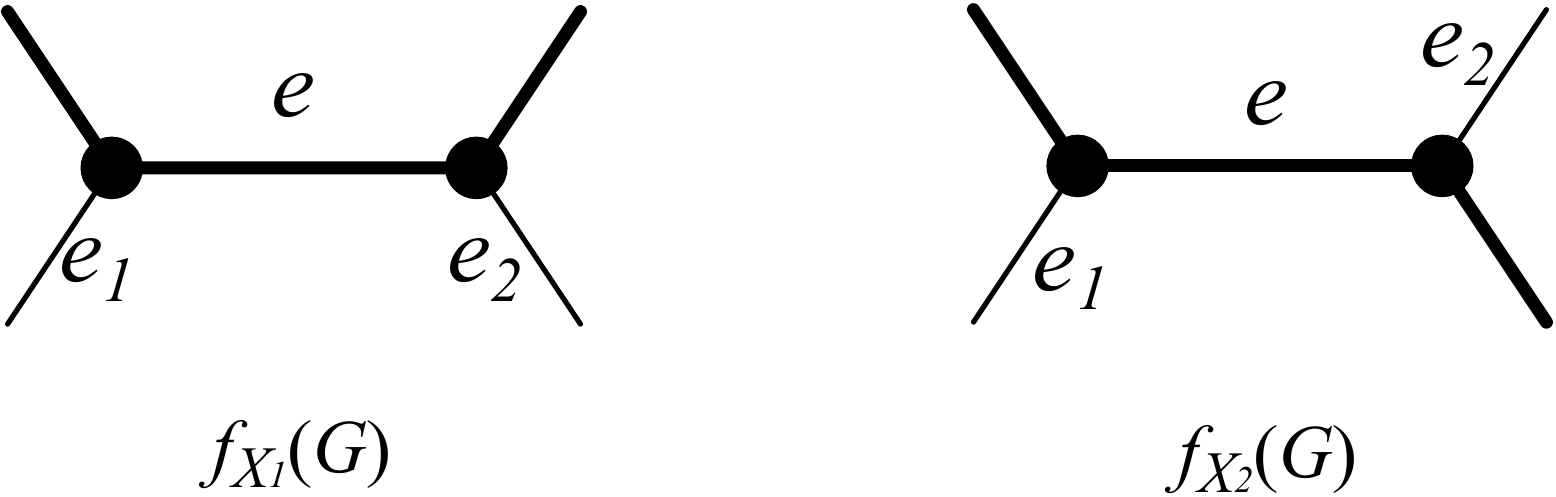}
\end{center}
\caption{The neighbourhoods around $e$ in $f_{X_1}(G)$ and $f_{X_2}(G)$}
\label{arounde}
\end{figure}

For the walk $W=e_1ee_2$ of $G$,
$f_{X_1}(W)$ constructs a consecutive part of a facial walk in $f_{X_1}(G)$ on $F^2_{X_1}$
but $f_{X_2}(W)$ are not so on $F^2_{X_2}$.
Thus, $F^2_{X_1} \neq F^2_{X_2}$, or $f_{X_1}(G)$ and $f_{X_2}(G)$ are not equivalent.
It implies that the choice of a subset $X$ of $E(G)$ uniquely induces the re-embedding $f_X(G)$ of $G$ up to equivalence.
Moreover, the total number of inequivalent embeddings of $G$ is $2^{\vert E(G)\vert }$,
and $F^2_X$ is homeomorphic to the sphere
if and only if $X$ is empty.

In the situation shown in the right of Fig.~\ref{arounde},
we say that two cycles $f_{X_2}(C)$ and $f_{X_2}(C')$ \emph{cross} along an edge $e$.
Since $G$ is 3-connected and planar,
there are no edges and vertices
contained in both $C$ and $C'$
other than $e$ and its endvertices.
Thus, $f_{X_2}(C)$ and $f_{X_2}(C')$ cross exactly once.
Note that any two cycles in $f_X(G)$ for a given subset $X$ of $E(G)$ do not cross at a single vertex since $G$ is 3-regular.

Let $H$ be a subgraph of a graph $G'$.
An \emph{$H$-bridge} in $G'$ is a subgraph of $G'$
induced by either an edge not in $H$
but with both ends in $H$,
or a component of $G'-V(H)$
together with all edges joining it to $H$.
Note that any two $H$-bridges are edge-disjoint.
It has been known that for a facial cycle $C$ of a 3-connected planar graph $G$,
there is only one $C$-bridge in $G$ (e.g., see \cite[p.39--40]{tex2}).

\begin{lemma}\label{lem1}
Let $C$ be a facial cycle in $G$
and let $f_X(G)$ be a re-embedding of $G$ with a given subset $X$ of $E(G)$.
Then, $f_X(C)$ is a non-separating cycle on $F^2_X$ if and only if $f_X(C)$ has a twisted edge.
\end{lemma}

\begin{proof}
It is easy to see that if $f_X(C)$ has no twisted edges then it is facial in $f_X(G)$
and hence it separates $F^2_X$ into two regions.

Suppose that $f_X(C)$ has a twisted edge $f_X(e)$,
that is, an edge $e$ of $G$ is in $X$.
Let $C'$ be the other facial cycle of $G$
containing $e$.
As shown in the right of  Fig.~\ref{arounde},
$f_X(C)$ and $f_X(C')$ cross along $e$ and hence two edges of $f_X(C')$ adjacent to $f_X(e)$ are located separately in opposite sides of $f_X(C)$.
However, both of these edges are contained in the unique $f_X(C)$-bridge.
This implies that $f_X(C)$ does not separate $F_X^2$.
\end{proof}

\section{Characterizations of re-embedding structures}\label{sec2}

In this section, we shall characterize the structures of $f_X(G)$
when $F^2_X$ is homeomorphic to the projective-plane, the torus or the Klein bottle,
and show the following theorems.

\begin{theorem}\label{thpro}
A $3$-connected $3$-regular graph embedded on the projective-plane is planar
if and only if it has one of the two structures $(P1)$ and $(P2)$ shown in Fig.~$\ref{pro}$.
\end{theorem}

\begin{figure}[htbp]
\begin{center}
\includegraphics[width=80mm]{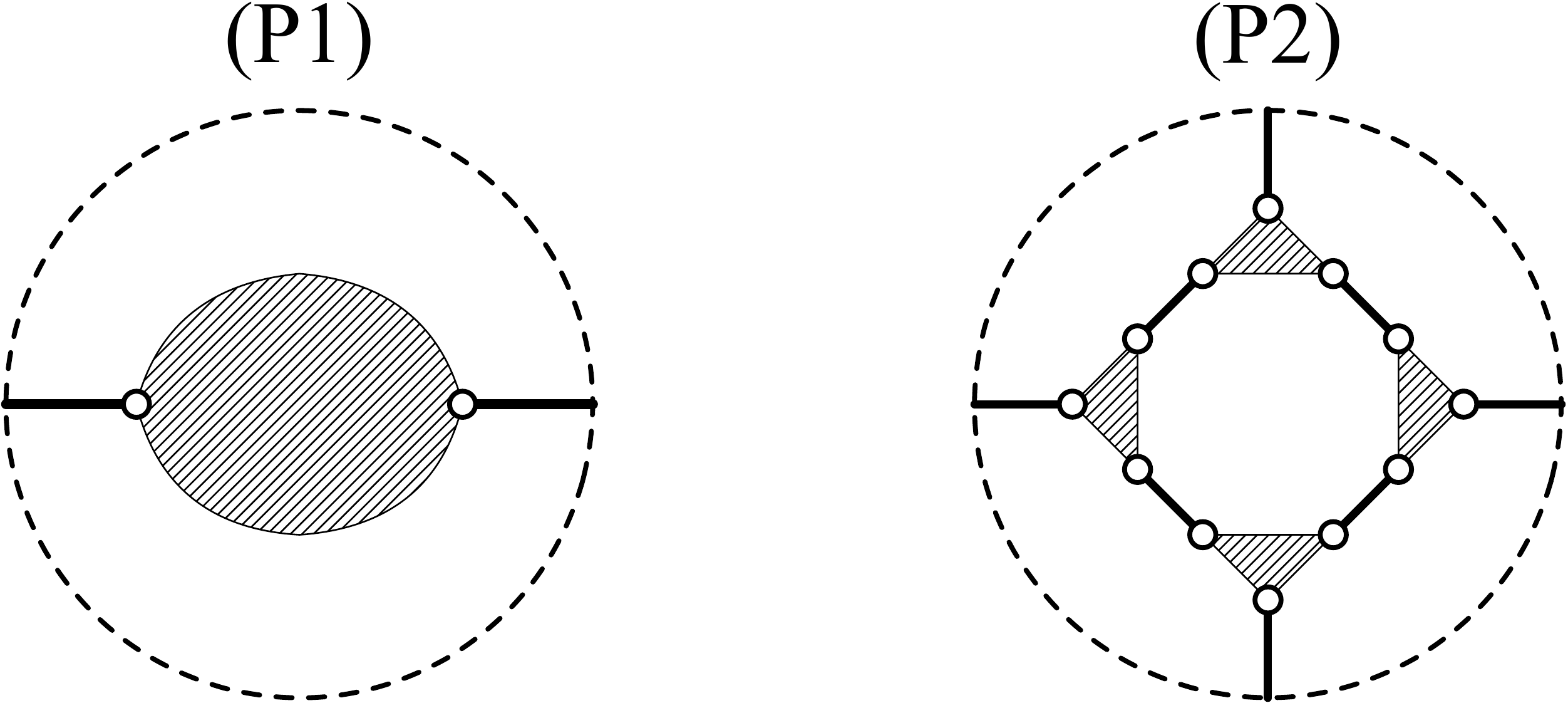}
\end{center}
\caption{Re-embedding structures on the projective-plane}
\label{pro}
\end{figure}

\begin{theorem}\label{thtorus}
A $3$-connected $3$-regular graph embedded on the torus is planar
if and only if it has one of the two structures $(T1)$, $(T2)$ and $(T3)$ shown in Fig.~$\ref{torus}$.
\end{theorem}

\begin{figure}[htbp]
\begin{center}
\includegraphics[width=110mm]{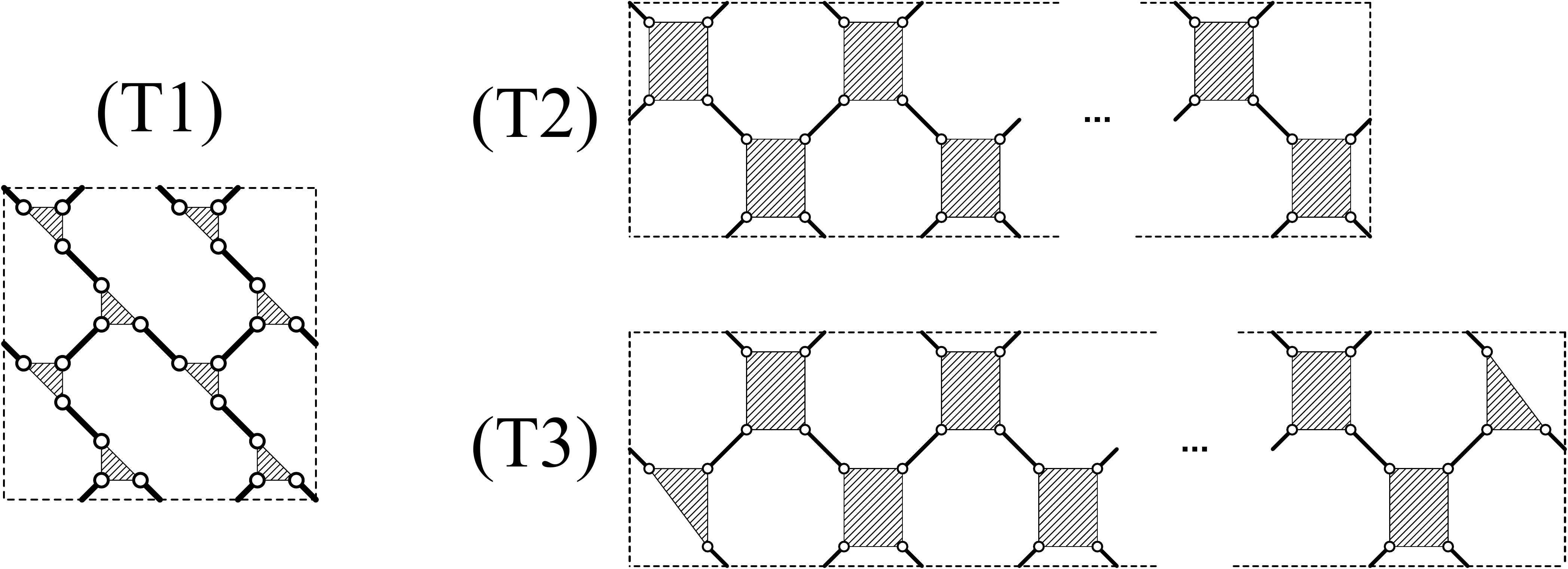}
\end{center}
\caption{Re-embedding structures on the torus}
\label{torus}
\end{figure}

\begin{theorem}\label{thKle}
A $3$-connected $3$-regular graph embedded on the Klein bottle is planar
if and only if it has one of the eight structures $(K1)$ to $(K8)$ shown in Fig.~$\ref{Kle}$.
\end{theorem}

\begin{figure}[htb]
\begin{center}
\includegraphics[width=100mm]{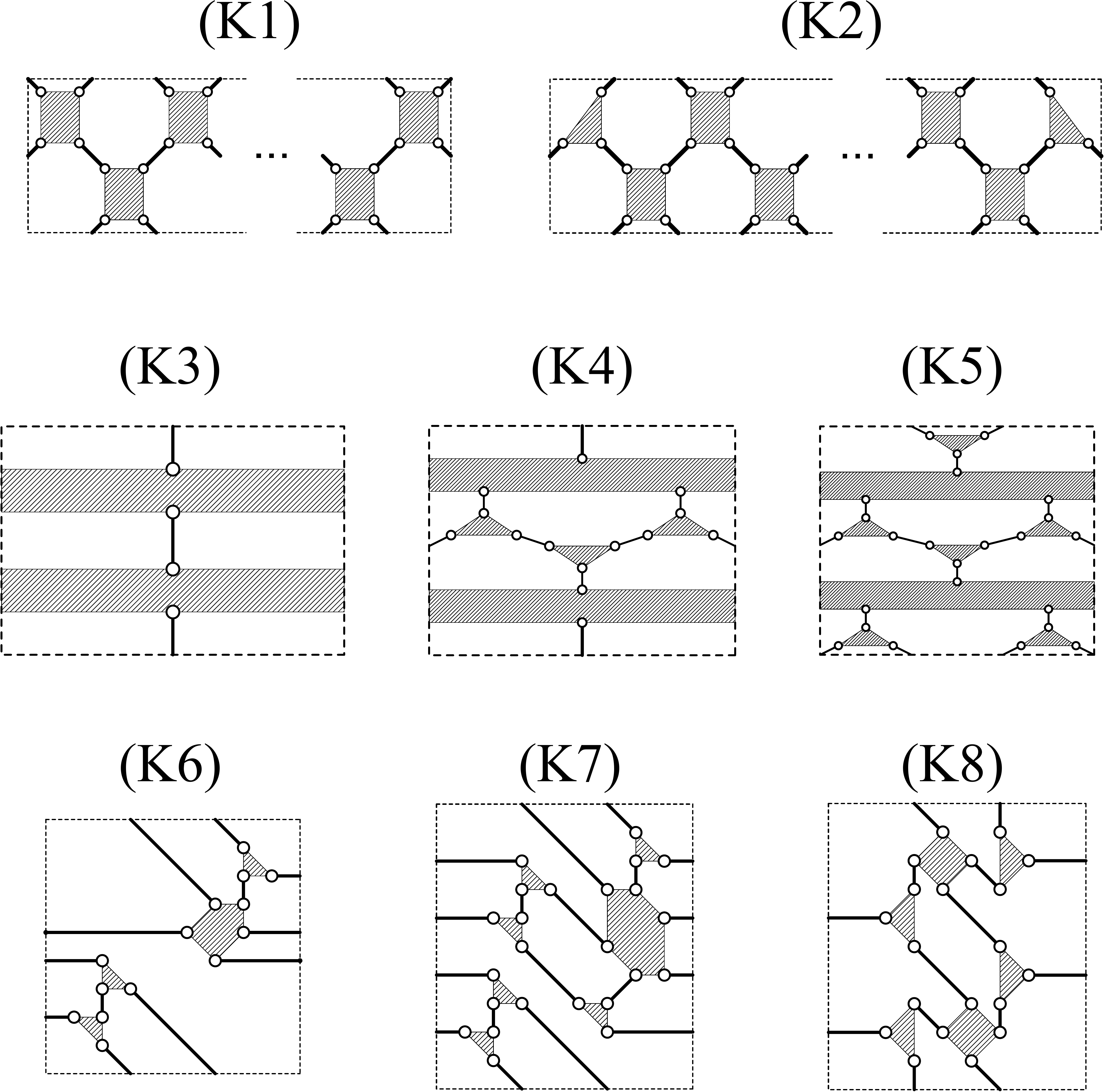}
\end{center}
\caption{Re-embedding structures on the Klein bottle}
\label{Kle}
\end{figure}

In Fig.~\ref{pro},
each pair of antipodal points on the dashed circle should be identified to recover the projective-plane.
Similarly, in Fig.~\ref{torus},
to recover the torus,
both pairs of opposite sides of dashed rectangle should be identified in the same direction,
and in Fig.~\ref{Kle},
to recover the Klein bottle,
the top and bottom sides of the dashed rectangle should be identified in the same direction
while the left and right sides should be identified in the opposite direction.
In these figures,
each of shaded areas corresponds to a component of the graph obtained form the original graph
by deleting all edges drawn by bold lines.
Some vertices on the boundary of such an area may not be different from each other,
that is, the edges drawn by bold lines may not be disjoint.
We omit a series of shaded rectangles from $(T2)$, $(T3)$, $(K1)$ and $(K2)$.
Both $(T2)$ and $(T3)$ have an even number of shaded rectangles
($(T3)$ may have no shaded rectangle),
while both $(K1)$ and $(K2)$ have an odd number of shaded rectangles.

In \cite{Mohar3},
the re-embedding structure of $2$-connected planar graphs on the projective-plane was analysed in detail (see Theorem 3.2 in \cite{Mohar3}),
while we focus on $3$-connected $3$-regular planar graphs.
Then, our re-embedding structure on the projective-plane,
shown in Theorem \ref{thpro},
is a special case in \cite{Mohar3},
and follows from it.
However, our proof is very simple
and important for us to understand other Theorems (e.g. Theorem \ref{th1}).
We thus provide a full proof of Theorem \ref{thpro}.
Moreover, we also construct the re-embedding structures on the torus and the Klein bottle,
which is not characterized completely in \cite{Mohar2,Mohar3}.
One may think that the case of the Klein bottle can be easily obtained from the case of projective-plane,
but this is not true.
Some structures in Theorem \ref{thKle} (e.g. $(K3)$, $(K4)$ and $(K5)$) can be regarded as simple combinations of the structure in Theorem \ref{thpro},
but some are not (e.g. $(K1)$ and $(K2)$).
\medskip

Let $H_X$ be the subgraph of the dual of $G$ (embedded on the sphere) induced by all edges dual to edges of the given subset $X$ of $E(G)$.
Then, there is a vertex of $H_X$ located in the inside of each face of $G$ whose facial cycle has an edge in $X$.
We shall specify what $H_X$ is isomorphic to
when $F^2_X$ is homeomorphic to the projective-plane, the torus or the Klein bottle,
which are essential ideas to prove not only Theorems \ref{thpro}, \ref{thtorus} and \ref{thKle} but also Theorems \ref{th1}, \ref{th0} and \ref{th2}.

First of all,
we give a simple condition of $H_X$
when $F^2_X$ is homeomorphic to a nonorientable surface.
It has been known that an embedding scheme
defines an embedding of a given graph on nonorientable surface
if and only if there is a cycle containing an odd number of twisted edges (see \cite[p.24--25]{tex2}).
It implies the following lemma.

\begin{lemma}\label{lem2}
For a given subset $X$ of $E(G)$,
$F^2_X$ is nonorientable
if and only if there is a vertex of odd degree in $H_X$.
\end{lemma}

\begin{proof}
If there is a vertex of odd degree in $H_X$
then the facial cycle, denoted by $C$,
corresponding to the vertex contains an odd number of edges in $X$.
Thus, $f_X(C)$ contains an odd number of twisted edges.

Suppose that $F^2_X$ is nonorientable.
Then, there is a cycle in $f_X(G)$ containing an odd number of twisted edges,
that is, there is a cycle in $G$
containing an odd number of edges in $X$,
denoted by $C'$.
Since $C'$ separates the sphere into two regions,
the edges dual to $X\cap E(C')$ form an edge-cut of $H_X$,
whose cardinality is odd.
Thus, there is a vertex of odd degree in $H_X$
by the handshaking lemma.
\end{proof}

\subsection{On the projective-plane}

\begin{lemma}\label{lempro}
For a given subset $X$ of $E(G)$,
$F^2_X$ is homeomorphic to the projective-plane if and only if
$H_X$ is isomorphic to $K_2$ or $K_4$.
\end{lemma}

\begin{proof}
Suppose that $F^2_X$ is homeomorphic to the projective-plane.
Any two non-separating simple closed curves on the projective-plane cross at least once.

Let $C$ and $C'$ be any two facial cycles in $G$
each of which has an edge of $X$.
By Lemma \ref{lem1},
$f_X(C)$ and $f_X(C')$ are non-separating cycles on $F^2_X$ and cross at most once.
Thus, $f_X(C)$ and $f_X(C')$ cross exactly once
and hence $C$ and $C'$ have exactly one common edge in $X$.
It implies that any two vertices in $H_X$ are adjacent to each other,
that is, $H_X$ must be a complete graph.
Since $H_X$ is planar and induced by edges,
$H_X$ must be isomorphic to $K_2$, $K_3$ or $K_4$.
However, $H_X$ is not isomorphic to $K_3$ by Lemma \ref{lem2}.

If $H_X$ is isomorphic to $K_2$ or $K_4$
then $G$ must have one of the structures shown in Fig.~\ref{proHX} ($H_X$ is drawn by squares and dashed lines).
Note that $K_4$ is uniquely embeddable on the sphere and hence the structure is determined uniquely.
In Fig.~\ref{proHX},
we represent edges in $X$ by bold lines
and each component of $G-X$ by shaded area together with some vertices,
each of which is an end vertex of an edge in $X$, on its boundary.
Note that $X$ does not have to be a matching,
that is, two edges in $X$ may have a common end vertex.

\begin{figure}[htbp]
\begin{center}
\includegraphics[width=100mm]{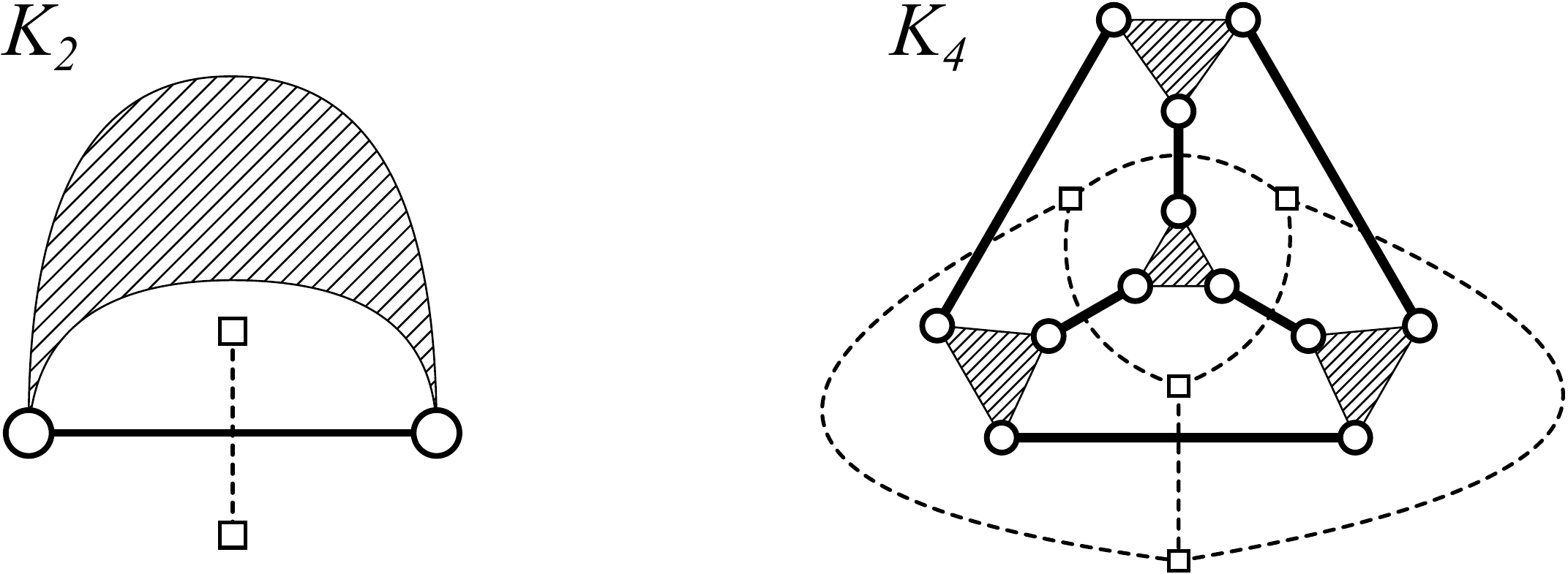}
\end{center}
\caption{Two structures of $G$ with $H_X$}
\label{proHX}
\end{figure}

In the situation shown in Fig.~\ref{proHX},
by twisting all edges of $X$,
we obtain the re-embedding $f_X(G)$ into the projective-plane shown in Fig.~\ref{pro}.
\end{proof}

\begin{proof}[of Theorem \ref{thpro}]
Let $G$ be a 3-connected 3-regular planar graph.
Any embedding of $G$ on a non-spherical surface can be represented by $f_X(G)$ with a suitable non-empty subset $X$ of $E(G)$.
By lemma \ref{lempro},
if $f_X(G)$ is a re-embedding of $G$ into the projective-plane then it has one of the structures shown in Fig.~\ref{pro}.

Conversely, it is easy to see that
if a 3-connected 3-regular graph has one of the structures shown in Fig.~\ref{pro},
then it can be embedded on the sphere
so that it has one of the structures shown in Fig.~\ref{proHX}. 
\end{proof}

\subsection{On the torus}

\begin{lemma}\label{lemtorus}
For a given subset $X$ of $E(G)$,
$F^2_X$ is homeomorphic to the torus if and only if
$H_X$ is isomorphic to $K_{2,2,2}$, $K_{2,2m}$, or $K_{1,1,2m-1}$ for some positive integer $m$.
\end{lemma}

\begin{proof}
Suppose that $F^2_X$ is homeomorphic to the torus.
For two simple closed curves crossing at most once on the torus,
they cross if and only if they are not homotopic.

Let $C$ and $C'$ be any two facial cycles in $G$
each of which has an edge of $X$.
By Lemma \ref{lem1},
$f_X(C)$ and $f_X(C')$ are non-separating cycles on $F^2_X$ and cross at most once.
Then, $f_X(C)$ and $f_X(C')$ are homotopic
if and only if they do not cross,
that is, $C$ and $C'$ have no common edge in $X$,
and hence two vertices in $H_X$ corresponding to them are not adjacent.
It implies that $H_X$ must be a complete multipartite graph
and each partite set corresponds to a non-null homotopy class on the torus.

It is easy to check that any planar complete miltipartite graph is isomorphic to one of the $7$ graphs $K_{1,1,1,2}$, $K_{1,1,1,1}=K_4$, $K_{2,2,2},K_{1,2,2}$, $
K_{1,1,n}$, $K_{2,n}$ and $K_{1,n}$ for some natural number $n$.
By Lemma \ref{lem2}, any vertex of $H_X$ has even degree.
Then, as $H_X$ is planar,
$H_X$ is isomorphic to $K_{2,2,2}$, or $K_{2,2m}$ or $K_{1,1,2m-1}$ for some positive integer $m$.

Conversely, if $H_X$ is isomorphic to $K_{2,2,2},K_{2,2m}$ or $K_{1,1,2m-1}$,
then $G$ must have one of the structure shown in Fig.~\ref{torusHX}.
Note that all of $K_{2,2,2},K_{2,2m}$ and $K_{1,1,2m-1}$ is uniquely embeddable on the sphere if we neglect the labels of their vertices.

\begin{figure}[htbp]
\begin{center}
\includegraphics[width=120mm]{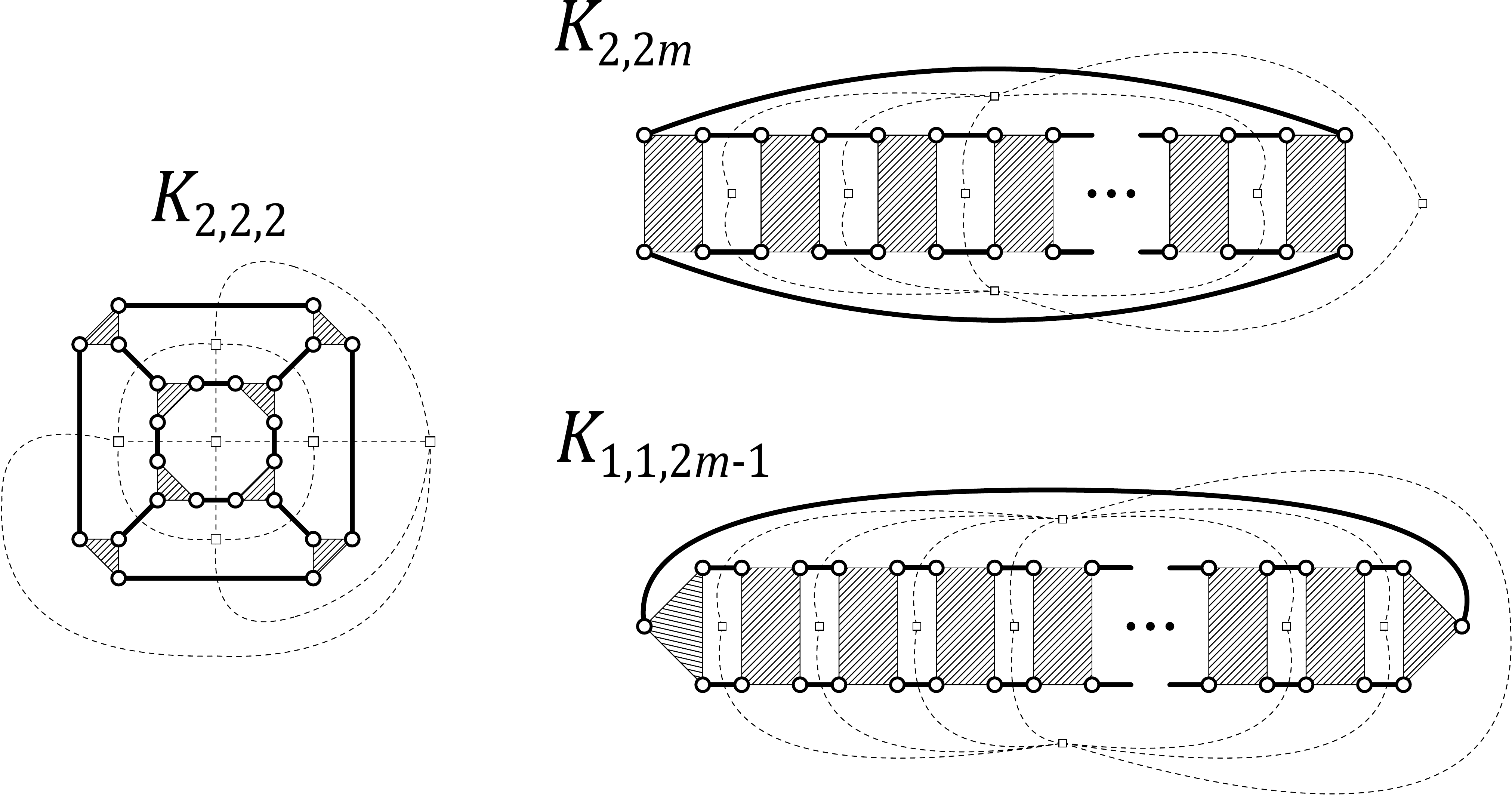}
\end{center}
\caption{Three structures of $G$ with $H_X$}
\label{torusHX}
\end{figure}

In the situation shown in Fig.~\ref{torusHX},
by twisting all edges of $X$,
we obtain the re-embedding $f_X(G)$ into the torus shown in Fig.~\ref{torus}.

\end{proof}

\begin{proof}[of Theorem \ref{thtorus}]
Like Theorem \ref{thpro},
this theorem follows immediately from the key lemma; Lemma \ref{lemtorus}. 
\end{proof}

\subsection{On the Klein bottle}

A simple closed curve on a surface is said to be \emph{$2$-sided}
if it divides its annular neighbourhood into two parts,
and to be \emph{$1$-sided} otherwise.

\begin{lemma}\label{lemKle}
For a given subset $X$ of $E(G)$,
$F^2_X$ is homeomorphic to the Klein bottle if and only if
$H_X$ is isomorphic to $K_{2,2m-1}$ or $K_{1,1,2m}$ for some positive integer $m$, or one of the six graphs $A_1$ to $A_6$ shown in Fig.~$\ref{HX}$.
\end{lemma}

\begin{proof}
Suppose that $F^2_X$ is homeomorphic to the Klein bottle.
There are exactly two mutually disjoint non-separating simple closed $1$-sided curves 
and exactly one non-separating $2$-sided curve
on the Klein bottle up to homotopy.

Let $C$ and $C'$ be any two facial cycles in $G$
each of which has an edge of $X$.
Then, $f_X(C)$ and $f_X(C')$ are non-separating cycles on $F^2_X$
and cross at most once.
We first assume that both $f_X(C)$ and $f_X(C')$ are $1$-sided.
Then, $f_X(C)$ and $f_X(C')$ cross
if and only if they are homotopic.
Second, we assume that one of $f_X(C)$ and $f_X(C')$ is $1$-sided and the other is $2$-sided.
Then, they cross.
Third, we assume that both $f_X(C)$ and $f_X(C')$ are $2$-sided.
Then they are homotopic and hence do not cross.

The vertex of $H_X$ corresponding to $C$ has odd degree if and only if $f_X(C)$ is $1$-sided.
Thus, the facts mentioned in the last paragraph imply that $H_X$ has the following conditions.
(1) The vertices of odd degree in $H_X$
induce a graph having at most two components
each of which is isomorphic to a complete graph.
(2) Any vertex of even degree and any vertex of odd degree are adjacent.
(3) The vertices of even degree in $H_X$
are independent,
that is, any pair of such vertices are not adjacent.

Let $V_{odd}$ (resp. $V_{even}$) be the set of vertices of odd (resp. even) degree in $H_X$.
Since $V_{even}$ is a independent set and any vertex of $V_{even}$ is adjacent to each vertex of $V_{odd}$,
$|V_{odd}|$ is even.
\medskip

\textbf{Case 1}: $V_{odd}$ induces a complete graph $K_m$.
As $H_X$ is planar, $m=2$ or $4$.

\textbf{Subcase 1a}: $m=2$.
It is easy to see that $|V_{even}|$ is even.
Then, $H_X$ is isomorphic to $K_{1,1,2k}$ with some non-negative integer $k$.
However, if $H_X$ is isomorphic to $K_{1,1,0}=K_2$
then $F^2_X$ is homeomorphic to the projective-plane by Lemma \ref{lempro}.
Then, $k\geq 1$.

\textbf{Subcase 1b}: $m=4$.
If there is at least one vertex in $V_{even}$
then $H_X$ is not planar
since it contains $K_5$ as a subgraph.
Moreover, if $V_{even}$ is empty,
then $H_X$ is isomorphic to $K_4$
and hence $F^2_X$ is homeomorphic to the projective-plane by Lemma \ref{lempro}.
Therefore, $m\neq 4$.

\textbf{Case 2}: $V_{odd}$ induces two disjoint complete graphs $K_m$ and $K_n$.
Then, we have $m+n=2,4,6,8$.

\textbf{Subcase 2a}: $m+n=2$,
that is, $m=n=1$.
In this situation,
$H_X$ is isomorphic to $K_{2,2k-1}$ with some positive integer $k$.

\textbf{Subcase 2b}: $m+n=4$,
that is, $m=n=2$ or $m=1,n=3$.
Suppose that $m=n=2$.
If $|V_{even}|\geq 3$
then $H_X$ is not planar 
since it contains $K_{3,3}$ as a subgraph.
If $|V_{even}|=1$
then each vertex of $H_X$ has even degree,
which contradicts Lemma \ref{lem2}.
If $|V_{even}|=0$ or $2$
then $H_X$ corresponds to $A_1$ or $A_6$, respectively.

Suppose that $m=1$ and $n=3$.
If $|V_{even}|\geq 3$
then $H_X$ is not planar 
since it contains $K_{3,3}$ as a subgraph.
If $|V_{even}|=0,2$
then each vertex of $H_X$ has even degree,
which contradicts Lemma \ref{lem2}.
If $|V_{even}|=1$
then $H_X$ corresponds to $A_4$.

\textbf{Subcase 2c}: $m+n=6$,
that is, $m=n=3$ or $m=2,n=4$.
Suppose that $m=n=3$.
If $|V_{even}|\geq 3$
then $H_X$ is not planar 
since it contains $K_{3,3}$ as a subgraph.
If $|V_{even}|=0,2$
then each vertex of $H_X$ has even degree,
which contradicts Lemma \ref{lem2}.
If $|V_{even}|=1$
then $H_X$ corresponds to $A_5$.

Suppose that $m=2,n=4$.
If $|V_{even}|\geq 1$
then $H_X$ is not planar
since it contains $K_5$ as a subgraph.
If $|V_{even}|=0$
then $H_X$ corresponds to $A_2$.

\textbf{Subcase 2d}: $m+n=8$,
that is, $m=n=4$.
If $|V_{even}|\geq 1$
then $H_X$ is not planar 
since it contains $K_5$ as a subgraph.
If $|V_{even}|=0$
then $H_X$ corresponds to $A_3$.
\medskip

According to the above results,
$H_X$ is isomorphic to $K_{2,2m-1}$ or $K_{1,1,2m}$ for some positive integer $m$, or $H_X$ is isomorphic to one of the six graphs $A_1$ to $A_6$.

Conversely, if $H_X$ is isomorphic to $K_{2,2m-1}$ or $K_{1,1,2m}$ for some positive integer $m$, or $H_X$ is isomorphic to one of the six graphs $A_1$ to $A_6$,
then $G$ must have one of the structure shown in Fig.~\ref{KleHX}.
Note that all graphs shown in Fig.~\ref{HX} are uniquely embeddable on the sphere
if we neglect the labels of their vertices.

\begin{figure}[htb]
\begin{center}
\includegraphics[width=110mm]{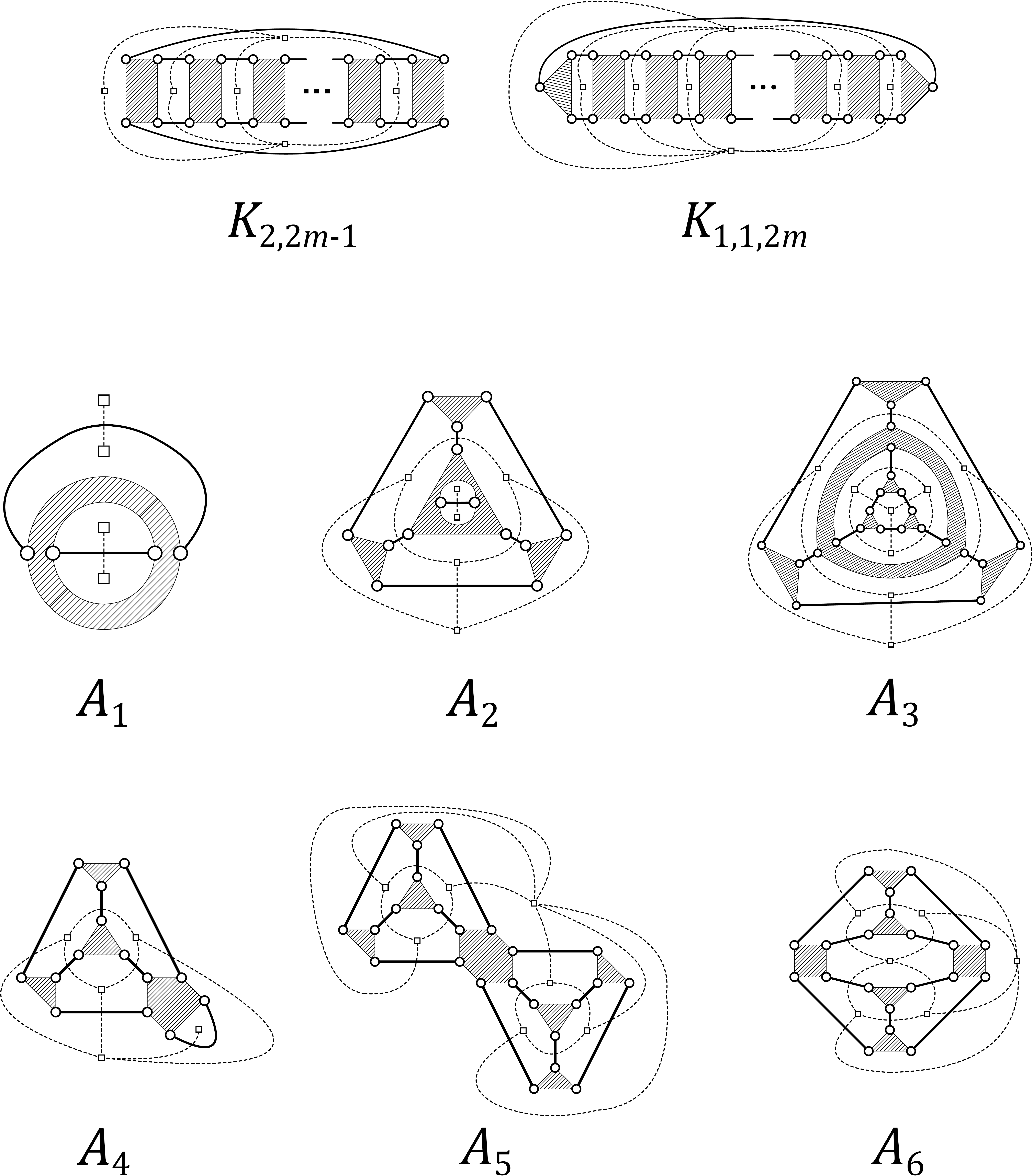}
\end{center}
\caption{Eight structures of $G$ with $H_X$}
\label{KleHX}
\end{figure}

In the situation shown in Fig.~\ref{KleHX},
by twisting all edges of $X$,
we obtain the re-embedding $f_X(G)$ into the Klein bottle shown in Fig.~\ref{Kle}.
\end{proof}

\begin{proof}[of Theorem \ref{thKle}]
Like Theorem \ref{thpro},
this theorem follows immediately from the key lemma; Lemma \ref{lemKle}.
\end{proof}

\subsection{Proof of Theorems}

Theorems \ref{th1}, \ref{th0} and \ref{th2} immediately follow from Lemmas \ref{lempro}, \ref{lemtorus} and \ref{lemKle}, respectively.

\begin{proof}[Proof of Theorems \ref{th1}, \ref{th0} and \ref{th2}]
Let $G$ be a $3$-connected $3$-regular planar graph.
Any embedding of $G$ on any surface is equivalent to an embedding $f_X(G)$ associated with a suitable subset $X$ of $E(G)$.
Moreover, such $X$ is unique.
Thus, Lemmas \ref{lempro}, \ref{lemtorus} and \ref{lemKle} imply Theorems \ref{th1}, \ref{th0} and \ref{th2}, respectively.
\end{proof}

\section{Inequivalent embeddings}\label{sec3}

In this section, we first give explicit bounds for the number of inequivalent embeddings of $G$ on each of the projective-plane, the torus and the Klein bottle.
After that, we propose algorithms for enumerating and counting these embeddings.

\subsection{The number of inequivalent embeddings}

Based on Theorems \ref{th1}, \ref{th0} and \ref{th2},
we show the following three results.

\begin{theorem}\label{numpro}
A $3$-connected $3$-regular planar graph with $n$ vertices has at least $\frac{3}{2}n$ and at most $2n-1$ inequivalent embeddings on the projective-plane.
\end{theorem}

\begin{theorem}\label{numtorus}
A $3$-connected $3$-regular planar graph with $n\geq 5$ vertices has at least $\frac{5}{2}n$ inequivalent embeddings on the torus.
\end{theorem}

\begin{theorem}\label{numKle}
A $3$-connected $3$-regular planar graph with $n$ vertices has at least $\frac{3}{8}n(3n+2)$ inequivalent embeddings on the Klein bottle.
\end{theorem}

Before we prove these theorems,
we consider a situation where
the dual of $G$ embedded on the sphere has many subgraphs isomorphic to $K_4$,
which is useful for showing the upper bound of Theorem \ref{numpro} and characterizing graphs attaining this bound.

A \emph{triangulation} on a surface
is an embedding of a graph on the surface
such that each face is bounded by a cycle of order $3$ and any two faces are incident with at most one common edge.
A graph embedded on the sphere is $3$-connected and $3$-regular if and only if the dual is a triangulation on the sphere.
For a triangulation $T$,
a \emph{$3$-vertex addition} is an operation of adding a vertex into a face $\Delta$ of $T$ and joining the new vertex to the vertices on the boundary of $\Delta$.

\begin{lemma}\label{lemup}
Every triangulation $T$ on the sphere has at most $ \left( |V(T)|-3 \right)$ subgraphs isomorphic to $K_4$.
In particular, $T$ attains the upper bound
if and only if
$T$ is obtained from $K_4$ embedded on the sphere
by a sequence of $3$-vertex additions. 
\end{lemma}

\begin{proof}
The proof is by induction on the number of vertices.

If $|V(T)|=4$
then $T$ is $K_4$ itself
and hence the result clearly holds.
Thus, we assume $|V(T)|\geq 5$.

If $T$ has no separating cycle of order $3$
then $T$ has no subgraph isomorphic to $K_4$.
Thus, we may assume that $T$ has a separating cycle $C$ of order $3$,
which separates the sphere into two regions,
denoted by $R_1$ and $R_2$.
Let $T_1$ (resp. $T_2$) be the subgraph of $T$
induced by the vertices lying on $R_1$ (resp. $R_2$) with its boundary.
Then, both of $T_1$ and $T_2$ is also a triangulation on the sphere.
Note that $T_1 \cap T_2 = C$ and $T_1 \cup T_2 = G$.
For any vertices $x\in V(T_1)\backslash V(C)$ and $y\in V(T_2)\backslash V(C)$,
there is no edge whose endvertices are $x$ and $y$,
and hence there are no subgraphs of $T$ isomorphic to $K_4$ having both $x$ and $y$.
Thus,
the number of subgraphs of $T$ isomorphic to $K_4$ is at most
$$ \left( |V(T_1)|-3 \right)
+ \left( |V(T_2)|-3 \right)=\left(|V(T)|+3\right)-6=|V(T)|-3.$$

Next, we characterize triangulations attaining this upper bounds.
Let $\tilde{T}$ be a triangulation on the sphere obtained from $T$ by one operation of a $3$-vertex addition
and $\tilde{v}$ be the additional vertex of $\tilde{T}$.
There is exactly one subgraph of $\tilde{T}$ isomorphic to $K_4$ including $\tilde{v}$.
If $T$ is obtained from $K_4$ by a sequence of $3$-vertex addition
then $T$ has has exactly $|V(T)|-3$ subgraphs isomorphic to $K_4$,
and hence $\tilde{T}$ has exactly $|V(T)|-2=|V(\tilde{T})|-3$ subgraphs isomorphic to $K_4$.

Conversely, suppose that $T$ has exactly $|V(T)|-3$ subgraphs isomorphic to $K_4$.
We may assume that $|V(T)|\geq 5$ and $T$ has a separating cycle $C$ of order $3$.
Then, $T_1$ and $T_2$,
which are defined in the same way as above,
must have exactly $|V(T_1)|-3$ and $|V(T_2)|-3$ subgraphs isomorphic to $K_4$, respectively,
and hence both are obtained from $K_4$ by a sequence of $3$-vertex additions.

Let $T'$ be a triangulation on the sphere obtained from $K_4$ by a sequence of $3$-vertex addition but not $K_4$.
It is easy to check that
any two vertices of degree $3$ are not adjacent in $T'$.
Thus, an operation of a $3$-vertex addition from $T'$ will not decrease the number of vertices of order $3$,
and hence $T'$ has at least two vertices of degree $3$.

The above facts imply that we can obtain $K_4$ from $T_2$ by deleting a vertex of degree $3$ without deleting the vertices on $C$.
By applying these operations to $T$
and deleting the last vertex from $R_2$,
we have just obtained $T_1$.
Therefore, $T$ is also obtained from $K_4$
by a sequence of $3$-vertex additions.
\end{proof}

\begin{proof}[of Theorem \ref{numpro}]
Let $G$ be a 3-connected 3-regular planar graph  embedded on the sphere with $n$ vertices
and $G^*$ be its dual.
Choose an edge $e$ of $G$ and put $X=\{ e \}$.
Then, $H_X$ is isomorphic to $K_2$ and hence $f_X(G)$ is embedded on the projective-plane
by Lemma \ref{lempro}.
It implies that $G$ has at least $|E(G)|$ inequivalent embeddings on the projective-plane.
Since $G$ is $3$-regular,
we have $|E(G)|=\frac{3}{2}n$.

By Lemma \ref{lemup},
there are at most $|V(G^*)|-3$ subgraphs of $G^*$ isomorphic to $K_4$.
By Euler's formula,
$|V(G^*)|=(|V(G)|+4)/2=\frac{n+4}{2}$
and hence $G^*$ has at most $ \left( \frac{n+4}{2}-3 \right)$ subgraphs isomorphic to $K_4$.
Thus, by Theorem \ref{th1}, $G$ has at most $\frac{3}{2}n+\left(\frac{n}{2}-1\right)=2n-1$ inequivalent embeddings on the projective-plane.
\end{proof}

Not only $3$-connected $3$-regular planar graphs,
any $2$-connected graph $G$ has $|E(G)|$ inequivalent embeddings on the projective-plane
by twisting each edge of $G$.
Then, the assumptions on $3$-connectivity and $3$-regularity are not necessary in the lower bound in Theorem \ref{numpro}.
However, these assumptions are clearly necessary in the upper bound in Theorem \ref{numpro}.
In fact, we can easily construct non-$3$-connected or non-$3$-regular $2$-connected planar graphs
having exponentially many inequivalent embeddings on the projective-plane.

\begin{proof}[of Theorem \ref{numtorus}]
Let $G$ be a 3-connected 3-regular planar graph  embedded on the sphere with at least $5$ vertices
and $G^*$ be its dual.
Every face of $G^*$ is bounded by a cycle of order $3$ ($=K_{1,1,1}$)
and every edge of $G^*$ forms a chord of a cycle of order $4$ ($=K_{2,2}$)
since it is incident with just two triangle faces.
If $G$ is not isomorphic to $K_4$,
then there are no other chords in this cycle.
Thus, $G^*$ has at least $|V(G)|$ cycles of order $3$
and at least $|E(G)|$ cycles of order $4$ as subgraphs.
As $G$ is 3-regular, $|E(G)|=3n/2$.
Then, by Theorem \ref{th0},
$G$ has at least $n+3n/2=5n/2$ inequivalent embeddings on the torus.
\end{proof}

Note that a $3$-connected $3$-regular planar graph with at most $4$ vertices must be isomorphic to $K_4$,
which has exactly $7\leq 5\cdot 4/2$ inequivalent embeddings on the torus.

\begin{proof}[of Theorem \ref{numKle}]
Let $G$ be a 3-connected 3-regular planar graph  embedded on the sphere with $n$ vertices
and $G^*$ be its dual.
Choose two distinct edges $e_1$ and $e_2$ of $G$
and put $X=\{ e_1, e_2\}$.
Then, $H_X$ is isomorphic to $K_{2,1}$ or $A_1$.
By Lemma \ref{lemKle}, $f_X(G)$ is embedded on the Klein bottle.

In addition, we try to find subgraphs isomorphic to $K_{1,1,2}$ in $G^*$.
Since $G^*$ is a triangulation on the sphere,
every edge $e^*$ is incident with just two triangle faces.
The five edges bounding these faces induce subgraphs of $G^*$ isomorphic to $K_{1,1,2}$.
Then, $G^*$ has $|E(G^*)|$ subgraphs isomorphic to $K_{1,1,2}$.

These results imply that, by Theorem \ref{th2}, $G$ has at least $\binom{|E(G)|}{2}+|E(G^*)|=\frac{3}{8}n(3n+2)$ inequivalent embeddings on the Klein bottle.
\end{proof}

\subsection{Examples}

First, we characterize the graphs attaining the lower bound of Theorem \ref{numpro}.
By Theorem \ref{numpro}, the following clearly holds.

\begin{corollary}\label{corolowerpro}
A $3$-connected $3$-regular planar graph $G$ with $n$ vertices has exactly $\frac{3}{2}n$ inequivalent embeddings on the projective-plane
if and only if the dual of $G$ embedded on the sphere has no subgraph isomorphic to $K_4$.\qed
\end{corollary}

By this corollary,
we show the following two families of graphs attaining the lower bound of Theorem \ref{numpro}.

A graph is \emph{cyclically $k$-edge-connected} if there is no set of at most $k-1$ edges
such that the graph obtained by deleting these edges has at least two components having a cycle.

\begin{proposition}\label{prolowerpro}
A $3$-connected $3$-regular planar graph $G$ with $n\geq 5$ vertices has exactly $\frac{3}{2}n$ inequivalent embeddings on the projective-plane
if $G$ is bipartite or cyclically $4$-edge-connected.
\end{proposition}

\begin{proof}
We only have to show that the dual $G^*$ of $G$ embedded on the sphere has no subgraph isomorphic to $K_4$.

If $G$ is bipartite
then degree of each vertex of $G^*$ is even, that is, $G^*$ is a \emph{even} triangulation.
It is well-known that every even triangulation 
on the sphere is (vertex) $3$-colorable
and hence has no subgraph isomorphic to $K_4$.

If $G$ is cyclically $4$-edge-connected and $n\geq 5$,
then $G^*$ has no separating cycle of order $3$
and hence has no subgraph isomorphic to $K_4$.
\end{proof}

Second, we characterize the graphs attaining the upper bound of Theorem \ref{numpro}.
Towards this goal, we introduce a transforming operation of $G$,
which corresponds to a $3$-vertex addition in the dual of $G$.

Let $v$ be a vertex of $G$, and $u_1,u_2$ and $u_3$ be vertices adjacent to $v$.
A \emph{truncation} of a vertex $v$ in $G$ 
is an operation of replacing a small part around $v$ with a cycle of order $3$ shown in Fig.~\ref{trunc};
delete $v$ and add new vertices $v_1,v_2$ and $v_3$ together with six edges $u_1v_1,u_2v_2,u_3v_3,v_1v_2,v_2v_3$ and $v_3v_1$.

\begin{figure}[htbp]
\begin{center}
\includegraphics[width=85mm]{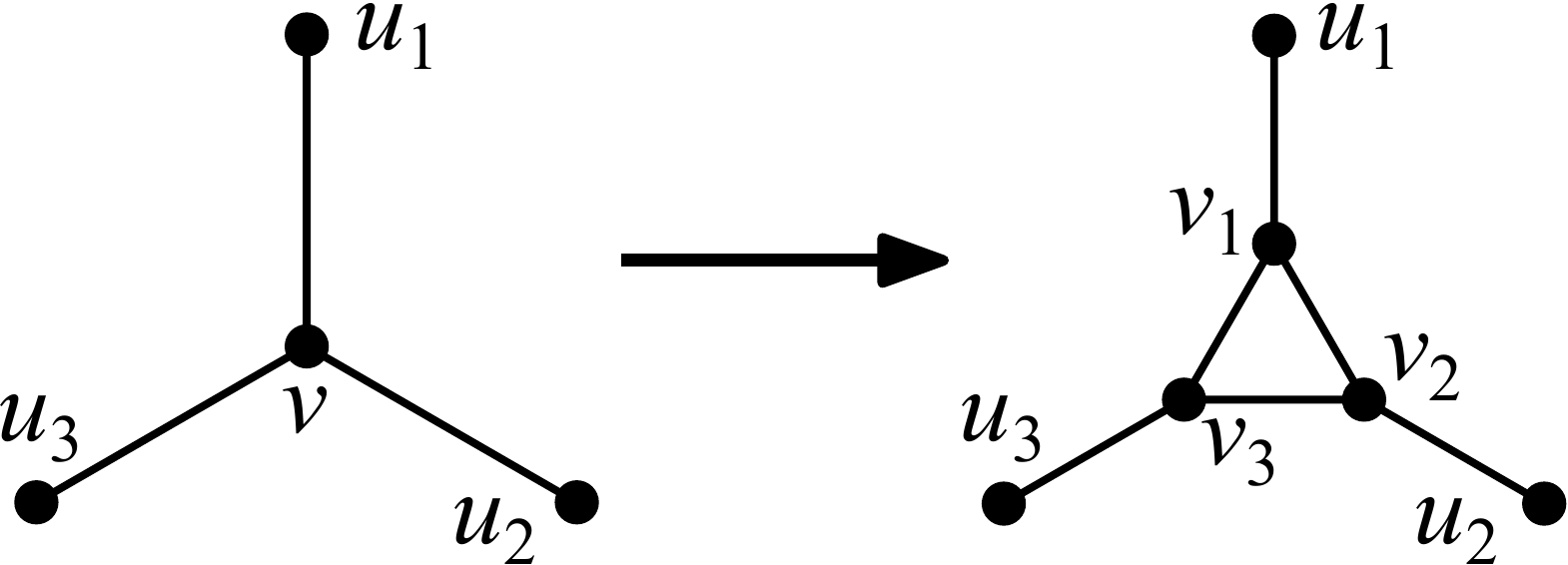}
\end{center}
\caption{Truncation of a vertex}
\label{trunc}
\end{figure}

The resulting graph,
denoted by $G'$,
is also $3$-connected, $3$-regular and planar.
The dual $(G')^*$ is obtained from the dual $G^*$ of $G$ by a $3$-vertex addition.
Then, the following clearly holds by Lemma \ref{lemup}.

\begin{corollary}\label{upperpro}
A $3$-connected $3$-regular planar graph $G$ with $n$ vertices has exactly $2n-1$ inequivalent embeddings on the projective-plane if and only if $G$ is obtained from $K_4$ by a sequence of truncations. \qed
\end{corollary}

Third, we provided graphs attaining the lower bounds of Theorems \ref{numtorus} and \ref{numKle}.

\begin{corollary}\label{corotorus}
A $3$-connected $3$-regular planar graph $G$ with $n\geq 5$ vertices has exactly $\frac{5}{2}n$ inequivalent embeddings on the torus
if and only if $G$ is cyclically $5$-edge-connected.
\end{corollary}

\begin{corollary}\label{coroKle}
A $3$-connected $3$-regular planar graph $G$ with $n$ vertices has exactly $\frac{3}{8}n(3n+2)$ inequivalent embeddings on the Klein bottle
if $G$ is cyclically $5$-edge-connected.
\end{corollary}

\begin{proof}[of Corollaries \ref{corotorus} and \ref{coroKle}]
Suppose that $G$ is cyclically $5$-edge-connected.
In each of the six structures shown in Fig.~\ref{torusHX} and Fig.~\ref{KleHX} corresponding to $K_{2,2,2},A_2,A_3,A_4,A_5$ and $A_6$,
we can easily find a set of at most four edges (drawn by bold lines)
such that the graph obtained from $G$ by deleting these edges has at least two components having a cycle. 
(The shaded annular area in $A_2$ must have a cycle.)
Then, $G$ has none of these six structures
and hence $G^*$ has no subgraph isomorphic to one of the six graphs $K_{2,2,2},A_2,A_3,A_4,A_5$ and $A_6$.

Suppose that $G^*$ has a subgraph isomorphic to $K_{2,n}$ or $K_{1,1,n}$ with $n\geq 3$.
Then, $G$ has one of the four structures shown in Fig.~\ref{torusHX} and Fig.~\ref{KleHX} corresponding $K_{2,2m},K_{1,1,2m-1},K_{2,2m-1}$ and $K_{1,1,2m}$.
In both case,
$G$ has at least three shaded areas,
and since $G$ is cyclically $5$-edge-connected,
all shaded areas except for at most one have no cycle.
Hence, there are two consecutive shaded areas having no cycle in $G$,
one of which is rectangle.
These shaded areas together with edges joining them form one of the three subgraphs shown in Fig.~\ref{shade}.

\begin{figure}[htbp]
\begin{center}
\includegraphics[width=105mm]{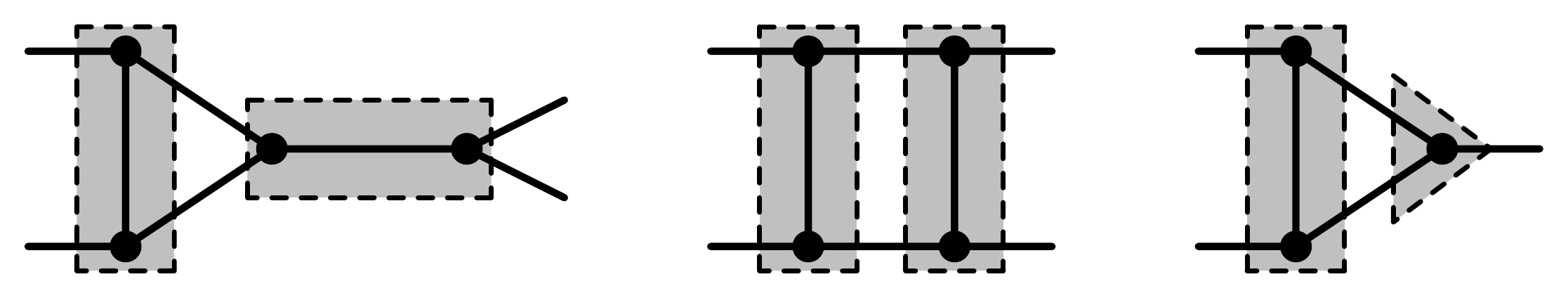}
\end{center}
\caption{Subgraphs formed by two consecutive shaded areas}
\label{shade}
\end{figure}

These subgraphs have a cycle and can be separated from $G$ by deleting at most four edges.
This implies that $G$ has at most one more shaded rectangle
and that it contains no cycle.
Hence it must be one of the three graphs shown in Fig.~\ref{prism}.
These graphs have distinct structures but each graph is the same as the others.
However, this graph is not cyclically $5$-edge-connected, a contradiction.

\begin{figure}[htbp]
\begin{center}
\includegraphics[width=105mm]{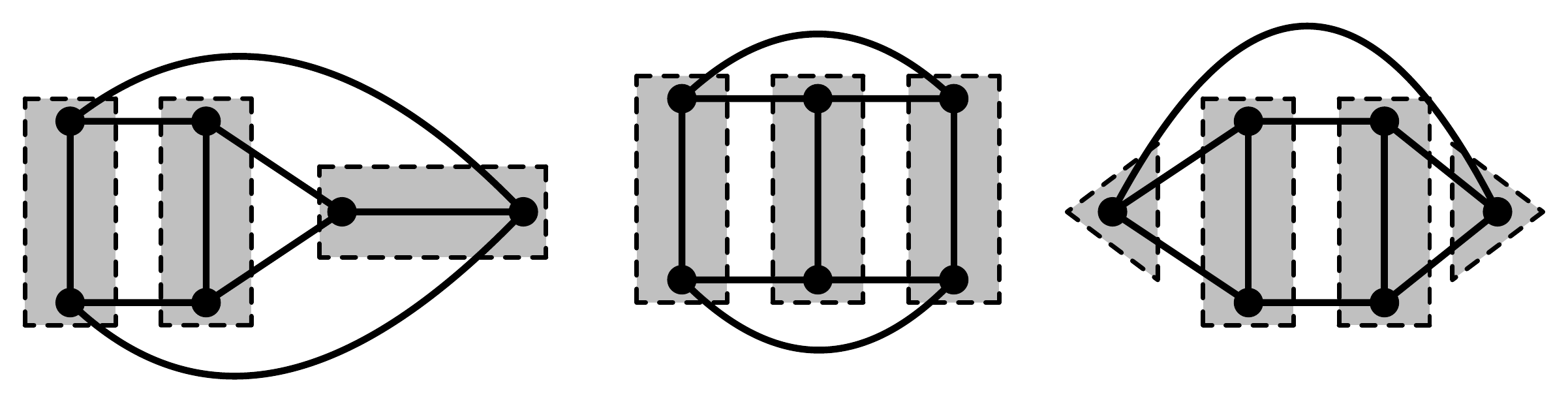}
\end{center}
\caption{Three graphs}
\label{prism}
\end{figure}

Therefore, $G^*$ has no subgraph isomorphic to $K_{2,n}$ or $K_{1,1,n}$ with $n\geq 3$,
and hence $G$ has only $\frac{5}{2}n$ inequivalent embeddings mentioned in the proof of Theorem \ref{numtorus}
on the torus,
and only $\frac{3}{8}n(3n+2)$ inequivalent embeddings mentioned in the proof of Theorem \ref{numKle}
on the Klein bottle.

Suppose $G$ is not cyclically $5$-edge-conneted.
Since $G$ is 3-connected,
there are three or four edges of $G$
whose removal results in a disconnected graph having exactly two components,
both of which contain a cycle.
Let $X$ be such edges.
Then, $H_X$ is isomorphic to $K_{1,1,1}=K_3$ or $K_{2,2}$,
and hence $f_X(G)$ is embedded on the torus.
However, this re-embedding conforms to none of re-embeddings mentioned in the proof of Theorem \ref{numtorus}.
Thus, the number of inequivalent embeddings of $G$ on the torus is more than $\frac{5n}{2}$.
\end{proof}

A graph attaining the lower bound of Theorem \ref{numKle} is not necessarily cyclically $5$-edge-connected.
For example, the following graph shown in Fig.~\ref{noc} is such a graph.
We can construct infinitely many such graphs
but we omit this here.

\begin{figure}[htbp]
\begin{center}
\includegraphics[width=40mm]{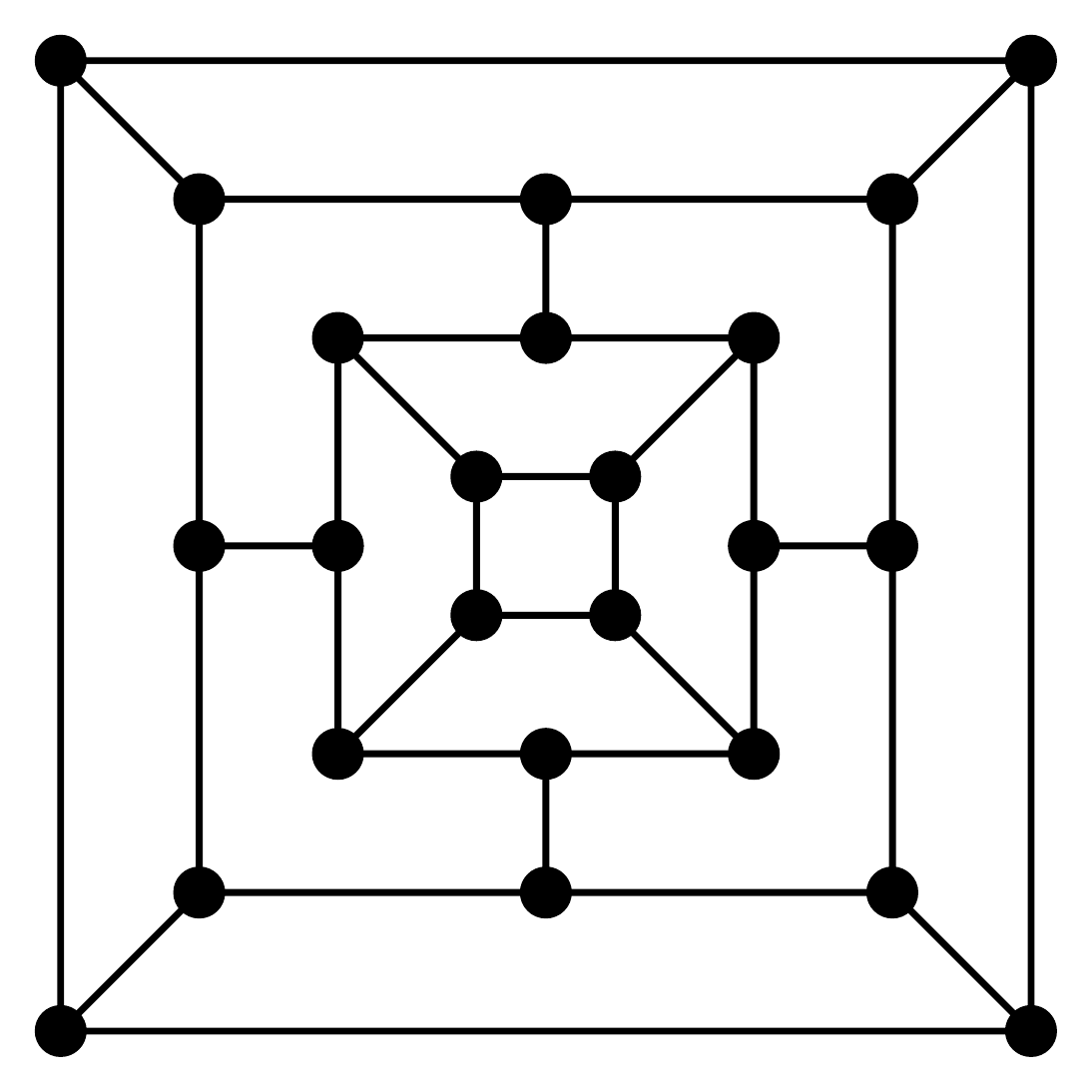}
\end{center}
\caption{A graph attaining the lower bound of Theorem \ref{numKle}}
\label{noc}
\end{figure}

Finally, we show graphs having exponentially many inequivalent embeddings on the torus and the Klein bottle.

\begin{proposition}\label{prop}
For a $3$-connected $3$-regular planar graph $G$ with $n$ vertices,
if the dual $G^*$ of $G$ embedded on the sphere
has a subgraph isomorphic to $K_{1,1,m}$
with a positive integer $m$
then $G$ has at least $2^m-1$ inequivalent embeddings on each of the torus and the Klein bottle.
\end{proposition}

\begin{proof}
For the complete tripartite graph $K_{1,1,m}$
with partite sets $V_1=\{u\},V_2=\{v\}$ and $V_3=\{a_1,a_2, \ldots ,a_m\}$,
let $S$ be any non-empty subset of $V_3$.
A subgraph induced by $V_1\cup V_2\cup S$
is isomorphic to $K_{1,1,|S|}$.
Deleting an edge $uv$ from this subgraph,
we obtain a subgraph isomorphic to $K_{2,|S|}$.

It implies that if $G^*$ has a subgraph isomorphic to $K_{1,1,m}$,
then we can find $2(2^{m}-1)$ subgraphs in $G^*$
isomorphic to $K_{1,1,k}$ or $K_{2,k}$ for some positive integer $k$.
Moreover, in these subgraphs,
the number of subgraphs isomorphic to $K_{1,1,k}$ is the same as the one isomorphic to $K_{2,k}$ for any $1\leq k\leq m$.
Then, by Theorems \ref{th0} and \ref{th2}, $G$ has at least $2^m-1$ inequivalent embeddings on each of the torus and the Klein bottle.
\end{proof}

\subsection{Algorithms}

By Theorem \ref{numpro},
the number of inequivalent embeddings of $G$ on the projective-plane is $O(n)$ with respect to the number $n$ of vertices of $G$.
In fact, we can easily enumerate these embeddings in polynomial-time with respect to $n$.
Note that we regard an enumeration of embedding  schemes of a graph 
as one of the embeddings of the graph.

\begin{theorem}\label{algo1}
There is a polynomial time algorithm for enumerating inequivalent embeddings of a $3$-connected $3$-regular planar graph on the projective-plane.
\end{theorem}

\begin{proof}
Let $G$ be a $3$-connected $3$-regular planar graph.
The embedding of $G$ on the sphere, its dual $G^*$ and another embedding $f_X(G)$ with a given subset $X$ of $E(G)$ can be obtained in polynomial time.
Then, we only have to find subgraphs isomorphic to $K_2$ or $K_4$ in $G^*$
by Lemma \ref{lempro},
which can be done in polynomial time.
\end{proof}

On the other hand,
there are $3$-connected $3$-regular planar graphs having exponentially many inequivalent embeddings on the torus and the Klein bottle
by Proposition \ref{prop}.
Then, we cannot enumerate inequivalent embeddings of such a graph on the torus or the Klein bottle in polynomial time.
However, we shall give a ``polynomial delay" algorithm for enumerating them.
An enumeration algorithm is said to be \emph{polynomial delay} if the maximum computation time between two consecutive outputs is polynomial in the input size.

For the complete miltipartite graphs $K_{2,m+2}$ and $K_{1,1,m+1}$ with any positive integer $m$,
there are exactly two vertices whose degree is not two.
We call them \emph{apex vertices}.

\begin{theorem}\label{algo2}
There is a polynomial delay algorithm for enumerating inequivalent embeddings of a $3$-connected $3$-regular planar graph on each of the torus and the Klein bottle.
\end{theorem}

\begin{proof}
Let $G$ be a $3$-connected $3$-regular planar graph.
We enumerate inequivalent embeddings of $G$ on the torus and the Klein bottle simultaneously.
Like the proof of Theorem \ref{algo1},
we only have to find subgraphs isomorphic to 
$K_{2,m+1}$ or $K_{1,1,m}$ for any positive integer $m$, or one of the seven graphs $A_1,\cdots , A_6$ and $K_{2,2,2}$.
(Every time we find such a subgraph, output a embedding corresponding to it)

First, we find subgraphs isomorphic to one of the nine graphs $A_1,\cdots ,A_6$, $K_{1,1,1}=K_3, K_{2,2}$ and $K_{2,2,2}$ in $G^*$,
which can be done in polynomial time.
Second, for a pair of vertices $u$ and $v$ of $G^*$,
enumerate vertices adjacent to both of them.
Third, enumerate subgraphs isomorphic to $K_{2,m+2}$ or $K_{1,1,m+1}$
whose apex vertices are $u$ and $v$.
In such a subgraph,
all non-apex vertices are already enumerated in the second step.
Thus, the third step can be done in polynomial delay time.

To repeat the second and third step for any pair of vertices,
in the end, we have just enumerated all subgraphs isomorphic to $K_{2,m+2}$ or $K_{1,1,m+1}$ in $G^*$.
\end{proof}

We can calculate the total number of inequivalent embeddings of $G$ on each of the projective-plane, the torus and the Klein bottle in polynomial time
by a simple improvement of the above algorithms.

\begin{corollary}\label{coroalgo}
There is a polynomial time algorithm for counting the number of inequivalent embeddings of a $3$-connected $3$-regular planar graph on each of the projective-plane, the torus and the Klein bottle.
\end{corollary}

\begin{proof}
The projective-planar case clearly holds.
We only have to count all the embeddings enumerated in the algorithm of Theorem \ref{algo1}.
Thus, we may consider embeddings on the torus and the Klein bottle.

Let $G$ be a $3$-connected $3$-regular planar graph.
On the basis of the algorithm in Theorem \ref{algo2},
we shall count subgraphs of $G^*$ isomorphic to
$K_{2,2,2},K_{2,2m}$ or $K_{1,1,2m-1}$,
which correspond to embeddings on the torus,
and count subgraphs isomorphic to $K_{2,2m-1}$ or $K_{1,1,2m}$, or one of the six graphs $A_1$ to $A_6$,
which correspond to embeddings on the Klein bottle.

Like the proof of Theorem \ref{algo2},
we first count the number of subgraphs of $G^*$ isomorphic to one of the three graphs $K_{2,2,2},K_{2,2}$ and $K_{1,1,1}=K_3$,
and denote it by $N_T$.
Similarly, we count the number of subgraphs isomorphic to one of the six graphs $A_1$ to $A_6$,
and denote it by $N_K$.

Second,
for the pair of vertices $u$ and $v$ in the proof of Theorem \ref{algo2},
assume that exactly $k$ vertices are adjacent to both $u$ and $v$.
Let $f_T(u,v)$ (resp. $f_K(u,v)$) be the number of subgraphs 
isomorphic to $K_{2,2m+2}$ or $K_{1,1,2m+1}$ (resp. $K_{2,2m-1}$ or $K_{1,1,2m}$) for any positive integer $m$
whose apex vertices are $u$ and $v$.
If $u$ and $v$ are adjacent to each other in $G^*$,
then we have
$$f_T(u,v)
=\sum_{i=3}^k \binom{k}{i}
=2^k-\frac{k(k-1)}{2}-k-1=2^k-\frac{k^2+k+2}{2}, $$
$$f_K(u,v)
=\sum_{i=1}^k \binom{k}{i}
=2^k-1.$$
Otherwise,
$$f_T(u,v)
=\sum_{i=2}^{\lfloor\frac{k}{2}\rfloor} \binom{k}{2i}
=2^{k-1}-\frac{k(k-1)}{2}-1=2^{k-1}-\frac{k^2-k+2}{2},$$
$$ f_K(u,v)
=\sum_{i=1}^{\lfloor\frac{k+1}{2}\rfloor} \binom{k}{2i-1}
=2^{k-1}.$$

Add the sum of $f_T(u,v)$ (resp. $f_K(u,v)$) taken over all pairs of vertices $u,v$ to $N_T$ (resp. $N_K$).
These are the total numbers of inequivalent embeddings of $G$ on the torus and the Klein bottle.
Thus, we can obtain this number in polynomial time.
\end{proof}

\section{Concluding remarks}

In this paper, we have shown the re-embedding structures of a $3$-connected $3$-regular planar graph $G$ on the projective-plane, the torus and the Klein bottle.
These structures enable us to count inequivalent embeddings of $G$ on each surface easily.

In order to extend our result to surfaces with higher genera,
we should show the complete lists of re-embedding structures of $G$ on these surfaces like Theorems \ref{thpro}, \ref{thtorus} and \ref{thKle}.
However, we think that
there are a large number of re-embedding types even on an orientable surface with genus $2$ or a nonorientable surface with genus $3$.
Then, it seems to be difficult to give such complete lists without additional assumptions.

\acknowledgements
\label{sec:ack}
The auther thanks anonymous referees for careful reading and their helpful comments which improved the quality of this paper.


\end{document}